\pdfoutput=1
\documentclass[11pt]{article}

\usepackage[utf8]{inputenc}
\usepackage[T1]{fontenc}
\usepackage{graphicx}
\usepackage{amsmath,amssymb,amsfonts}
\usepackage{amsthm}
\usepackage{mathtools}
\usepackage{enumitem}
\usepackage{microtype}
\usepackage{url}

\theoremstyle{plain}
\newtheorem{theorem}{Theorem}[section]
\newtheorem{lemma}[theorem]{Lemma}
\newtheorem{corollary}[theorem]{Corollary}

\newcommand{\keywords}[1]{\par\smallskip\noindent\textbf{Keywords:} #1}
\newcommand{\subclass}[1]{\par\smallskip\noindent\textbf{Mathematics Subject Classification:} #1}

\newcommand{\R}{\mathbb{R}}
\newcommand{\lstar}{\lambda^{\star}}
\newcommand{\Bac}{\mathrm{B\&C}}
\newcommand{\Bab}{\mathrm{B\&B}}
\newcommand{\eps}{\varepsilon}
\newcommand{\defeq}{\mathrel{:=}}

\begin{document}

\title{Sharp Logarithmic Thresholds for Cut Schedules in an Abstract Branch-and-Cut Model}

\author{Hongyi Jiang\\
Department of Systems Engineering\\City University of Hong Kong\\Hong Kong SAR, China\\
\texttt{hongyi.jiang@cityu.edu.hk}}

\date{}

\maketitle

\begin{abstract}
Branch-and-cut interleaves branching with cutting-plane
generation. How the two operations share the work of proving a bound is a
basic theoretical question. We study an abstract model in which a tree
certifies a target bound $Z$. Each branch node improves the bound by
$\ell$ on one child and by $r$ on the other, where $0<\ell\le r$. The
$i$th cut along a root-to-node path improves it by $c_i\ge0$, with
cumulative improvement $C_k=\sum_{i=1}^k c_i$.

Asymmetric branching enters through the rate $\lambda^{\star}>0$
defined by $e^{-\lambda^{\star}\ell}+e^{-\lambda^{\star}r}=1$. We
establish uniform two-sided bounds of order $e^{\lambda^{\star}Z}$ on the
minimal leaf count of pure branching trees. We then identify $\log k$ as the sharp threshold scale for the
power of cutting. For cut schedules with extended limit
$\gamma=\lim_{k\to\infty}C_k/\log k\in[0,\infty]$, minimal-size
trees obey a trichotomy. If $\gamma=\infty$, cuts prove
asymptotically all of the target. If $0\le\gamma<\infty$, the limiting
fraction of the bound proved by cuts is
$\gamma\lambda^{\star}/(1+\gamma\lambda^{\star})$. If $\gamma=0$,
branch-and-cut has the same exponential size rate as pure
branch-and-bound. This resolves open questions raised by Kazachkov, Le
Bodic, and Sankaranarayanan on minimal-size trees under
harmonically-worsening cuts, and generalizes their results to
asymmetric branching and to all cut schedules in the model with this logarithmic limit.
Finally, we show that branch-and-cut attains polynomial size in terms of $Z$ if and only if polynomially many
cuts reduce the residual bound to $O(\log Z)$.
\subclass{90C10, 90C11, 90C57}
\end{abstract}

\section{Introduction}

Branch-and-cut ($\Bac$) is the computational backbone of modern
mixed-integer optimization. It combines the branch-and-bound
($\Bab$) method of Land and Doig~\cite{land1960automatic} with the
cutting-plane method pioneered by Gomory~\cite{gomory1958outline}. This hybrid method
powers state-of-the-art solvers for a wide range of linear and
nonlinear optimization problems~\cite{junger200950,achterberg2013mixed}.
At each node of the search tree, the method interleaves two operations.
\emph{Branching} splits the node into two child subproblems.
\emph{Cutting} adds valid inequalities, called cutting planes or cuts,
that tighten the relaxation at the current node. Both operations
improve the bound certified by the tree for the optimal value. Pure
branch-and-bound is the special case in which no cuts are added.

Although the computational practice of $\Bac$ is highly
mature, the theory of how its two ingredients interact is relatively
recent. One active line of research develops abstract models of
branching that record only the bound improvement achieved at each
node~\cite{le2017abstract,anderson2021further}, while other works
prove lower bounds on the size of general $\Bab$
trees~\cite{dadush2020complexity,dey2023lower} and analyze the
theoretical properties of strong branching~\cite{dey2024theoretical}. On the
cutting side, the difficulties inherent in cutting-plane selection are
discussed in~\cite{dey2018theoretical}, the numerical behavior of pure
cutting-plane algorithms is examined
in~\cite{balas2010enumerative,zanette2011lexicography}, and cut
generation remains an active research
topic~\cite{dey2022cutting}, increasingly through learning-based
approaches~\cite{berthold2025learning,cheng2026generalization}. Closest to our theme, Basu,
Conforti, Di Summa, and Jiang compare the relative power of branching,
cutting, and their combination, proving in particular that $\Bac$ proof size can
be exponentially smaller than pure branching or pure cutting
alone~\cite{basu2023complexity,basu2022complexity}, while Shah, Dey, and Molinaro show that, under common full-strong-branching
scores, adding even a single cut can exponentially increase the size of
the resulting $\Bab$ tree~\cite{shah2026non}.

The abstract model of Kazachkov, Le Bodic, and
Sankaranarayanan~\cite{kazachkov2024abstract} isolates this interaction
by recording only how much each operation improves the bound. A $\Bac$
tree is a rooted binary tree with node set $V_T$. Each node $v$ carries
a label $z_T(v)\ge0$, the bound improvement accumulated from the root
to $v$, and $z_T(\mathrm{root})=0$. A node with two children is a
\emph{branch node}, and a node with one child is a \emph{cut node}. A
node with no children is a leaf. The tree \emph{proves} a target bound
$Z$ if $z_T(v)\ge Z$ at every leaf $v$. In the single-variable model
studied in~\cite{kazachkov2024abstract}, every branch node uses the
same pair of improvements $(\ell,r)$ with $0<\ell\le r$. Branching is
thus described by two numbers, while cut strengths may vary along a
path. This framework was recently extended by Han and
Kazachkov~\cite{han2026strength}, who study when placing all cuts at the
root, before any branching, is optimal or near-optimal under general
node-processing-time functions. In this paper, we study \emph{tree size}, the node
count $|V_T|$. In the terminology of~\cite{kazachkov2024abstract} this
is the time function $w=1$. 

For a node \(v\), define its \emph{residual bound} by
\(
        z(v)=Z-z_T(v),
\)
which is the part of the target that remains to be proved below \(v\).
We write \(z\) for \(z(v)\) when the node is clear from context. The
node is terminal once \(z\le 0\). A branch node with residual bound
\(z\) sends its two children to residual bounds \(z-\ell\) and
\(z-r\).

We first recall the symmetric case \(\ell=r\), where depth gives a
natural measure of tree size. To prove residual bound \(z\) by
branching alone, one needs depth \(\lceil z/r\rceil\), and hence
\(2^{\lceil z/r\rceil}\) leaves. Several results
in~\cite{kazachkov2024abstract} are stated in terms of this branching
depth. With asymmetric branching, \(\ell\neq r\), depth is no longer
the right invariant, because the two children of a branch node have
different residual bounds. The invariant that survives is the number of
leaves in a pure branching tree, which we now make precise.

Let \(L(z)\) be the minimum number of leaves in a pure branching tree,
one using no cuts, that proves residual bound \(z\). Then
\begin{equation}
        L(z)=L(z-\ell)+L(z-r)\text{ for }z>0,
        \text{ and }
        L(z)=1\text{ for }z\le 0,
        \label{eq:n-rec}
\end{equation}
and a pure branching tree with \(L(z)\) leaves has \(2L(z)-1\) nodes, so
\eqref{eq:n-rec} is the leaf-count form of the node-count recurrence of
Le Bodic and Nemhauser~\cite{le2017abstract}.
The exponential growth rate of \(L\) is governed by the unique number
\(\lambda^\star>0\) satisfying
\[
        e^{-\lambda^\star\ell}+e^{-\lambda^\star r}=1.
\]
Equivalently, $e^{\lambda^\star}$ is the \emph{growth ratio}
$\varphi>1$ of~\cite{le2017abstract}, i.e.\ the root of
$x^{r}-x^{r-\ell}-1$ larger than one, which admits no closed-form
expression in general~\cite{anderson2021further}. When $\ell=r$, the
defining equation gives $\lambda^\star=\log 2/r$, recovering the
symmetric depth rate.

Cuts enter through a schedule of strengths that depends only on
position, where the improvement of a cut is determined by how many cuts
precede it on its root-to-node path. Formally, the model is
parametrized by a sequence $c_1,c_2,\ldots$ with $c_i\ge0$, and in a
B\&C tree for $(\ell,r;c)$ the child of a cut node with $k$ cut nodes
strictly above it receives bound improvement $c_{k+1}$. Writing
$C_0=0$ and $C_k=\sum_{i=1}^k c_i$, a node with
$k$ cut nodes strictly above it thus accumulates exactly $C_k$ improvement from cutting; in particular,
$C_k$ is nondecreasing. Each cut node contributes one node to the size
$|V_T|$.
A \emph{minimal-size} tree is one with fewest nodes among all B\&C
trees for $(\ell,r;c)$ proving $Z$; one exists, since pure branching
proves $Z$ after at most $\lceil Z/\ell\rceil$ splits on every path
and node counts are positive integers.

This model leaves a fundamental quantitative question open:

\begin{center}
\itshape
How does a smallest $\Bac$ tree split the work between cutting and branching, and when do cuts change the proof's complexity class?
\end{center}

The remainder of the paper is organized as follows. In
Section~\ref{sec:contribution}, we summarize our main contributions.
Section~\ref{sec:branching-rate} analyzes pure branching under asymmetric
improvements and establishes the growth rate $\lstar$ of the minimal leaf
count. Section~\ref{sec:threshold} presents the classification theorem for
cut schedules, together with its specialization to constant and
harmonically-worsening cuts. Section~\ref{sec:dichotomy} characterizes when
branch-and-cut achieves polynomial size while pure branch-and-bound remains
exponential. Section~\ref{sec:conclusion} closes with concluding remarks.

\section{Contributions}\label{sec:contribution}

We answer the question posed above for all cut schedules whose cumulative strength has an extended logarithmic limit, sharpening the branching analysis it requires along the way.

First, Theorem~\ref{thm:branching-rate} refines the asymptotic growth
rate of~\cite{le2017abstract} to uniform two-sided bounds
$e^{\lstar z}\le L(z)\le K_{\ell,r}e^{\lstar z}$, valid for every
residual $z>0$ and for arbitrary real gains $0<\ell\le r$; this
non-asymptotic form is what the threshold arguments of
Sections~\ref{sec:threshold} and~\ref{sec:dichotomy} require. For
$\ell=r$ the rate reduces to the symmetric quantity $\log 2/r$ used
in~\cite{kazachkov2024abstract}.

Second, our main result (Theorem~\ref{thm:cut-price-classification} in
Section~\ref{sec:threshold}) classifies how a minimal-size tree divides the
target between cutting and branching. It applies whenever the extended limit
$\gamma=\lim_{k\to\infty}C_k/\log k\in[0,\infty]$ exists. The critical scale
is logarithmic in the number of cuts. The asymptotics are governed by the
single scalar $\gamma$. If $\gamma=\infty$, cuts prove asymptotically all of
the target. The branching component then shrinks to a vanishing fraction.
If $0<\gamma<\infty$, cutting and branching each prove a constant share,
with limiting cut fraction $\gamma\lstar/(1+\gamma\lstar)$. If $\gamma=0$,
the cut fraction vanishes. In that case B\&C gains no exponential-size
advantage over pure B\&B. This trichotomy resolves open questions raised
in~\cite{kazachkov2024abstract}. It also generalizes their results to
asymmetric branching and to all cut schedules in the model with this logarithmic limit. In particular,
their constant-cut theorem is extended from symmetric to asymmetric
branching. Their harmonic-cut limit was previously established for
approximately optimal trees and supported by numerical evidence for minimal
trees. We prove it holds for exact minimal-size trees
(Corollary~\ref{cor:recoveries}). 

Third, a complementary dichotomy (Theorem~\ref{thm:dichotomy} in
Section~\ref{sec:dichotomy}) characterizes when cutting changes the
complexity class of the proof. This is in the spirit of the exponential
separations between B\&C and its pure components established by Basu et
al.~\cite{basu2022complexity}. Pure B\&B remains exponential while B\&C has polynomial size if and only if some polynomially
bounded number of cuts reduces the residual bound to $O(\log Z)$. On the
scale of cut schedules, polynomial growth of $C_k$ suffices for this
speedup while harmonically-worsening cuts never achieve it
(Corollary~\ref{cor:poly-schedules}).

\section{The asymmetric branching rate}\label{sec:branching-rate}

This section establishes the quantitative form of the branching rate on
which the rest of the paper relies. In~\cite{le2017abstract}, the
growth ratio $\varphi=e^{\lstar}$ is obtained as an asymptotic limit
for integer gains and extended to rational gains by scaling.
Theorem~\ref{thm:branching-rate} instead gives explicit two-sided
bounds with the concrete constant $K_{\ell,r}=e^{\lstar r}$, valid for
every $z>0$ and for arbitrary real gains $0<\ell\le r$. The proofs of
Theorems~\ref{thm:cut-price-classification} and~\ref{thm:dichotomy}
also rely on this result.

\begin{theorem}
\label{thm:branching-rate}
Let $0<\ell\le r$, and let $\lstar>0$ be the unique solution of
\begin{equation}
        e^{-\lstar\ell}+e^{-\lstar r}=1.
        \label{eq:lambda-star}
\end{equation}
Take $K_{\ell,r}\defeq e^{\lstar r}$. Then for every $z>0$, we have
\begin{equation}
        e^{\lstar z}
        \le
        L(z)
        \le
        K_{\ell,r}e^{\lstar z}.
        \label{eq:growth}
\end{equation}
Consequently,
\(
        \log L(z) = \lstar z+O_{\ell,r}(1).
\)
\end{theorem}

\begin{proof}
By the definition of $\lstar$ in \eqref{eq:lambda-star}, multiplying
$e^{-\lstar\ell}+e^{-\lstar r}=1$ by $e^{\lstar z}$ yields
\begin{equation}
        e^{\lstar z}=e^{\lstar(z-\ell)}+e^{\lstar(z-r)}.
        \label{eq:exp-recur}
\end{equation}
To each residual $z$ we associate the level $m=\lceil z/\ell\rceil$, and we
argue by strong induction on $m$. The motivation is that when $z>0$
both child residuals lie at strictly smaller levels. In particular,
$\lceil(z-\ell)/\ell\rceil=m-1$, and since $z-r\le z-\ell$ (as $r\ge\ell$),
also $\lceil(z-r)/\ell\rceil\le m-1$.

For the lower bound of \eqref{eq:growth}, set
\[
        U(z)=
        \begin{cases}
        e^{\lstar z}, & z>0,\\
        1, & z\le0,
        \end{cases}
\]
which agrees with $L(z)$ on $z\le0$ by \eqref{eq:n-rec}. 

We first prove that $U(z)\le U(z-\ell)+U(z-r)$ for every
$z>0$. 
For each child residual $y\in\{z-\ell,z-r\}$ we have $e^{\lstar y}\le
U(y)$, since $e^{\lstar y}=U(y)$ when $y>0$ and $e^{\lstar y}\le1=U(y)$ when
$y\le0$; substituting both into \eqref{eq:exp-recur} gives $U(z)=e^{\lstar
z}=e^{\lstar(z-\ell)}+e^{\lstar(z-r)}\le U(z-\ell)+U(z-r)$. Let $P(m)$ be the assertion that $U(z)\le L(z)$ for
every $z$ with $\lceil z/\ell\rceil=m$. When $m\le0$ we have $z\le0$, so
$U(z)=1=L(z)$ by \eqref{eq:n-rec}, which settles the base of the induction.
Fix $m\ge1$ and assume $P(m')$ for all $m'<m$. Any $z$ at level $m$ has $z>0$,
and its children lie at levels below $m$, so the hypothesis gives
$U(z-\ell)\le L(z-\ell)$ and $U(z-r)\le L(z-r)$. Combining the
recurrence \eqref{eq:n-rec} with the recurrence inequality for $U$,
\[
        L(z)=L(z-\ell)+L(z-r)\ge U(z-\ell)+U(z-r)\ge U(z),
\]
which is $P(m)$. Hence $L(z)\ge U(z)=e^{\lstar z}$ for every $z>0$.

For the upper bound of \eqref{eq:growth}, let $Q(m)$ be the assertion that
$L(z)\le K_{\ell,r}e^{\lstar z}$ for every $z$ with $-r\le z$ and
$\lceil z/\ell\rceil=m$; proving $Q(m)$ for all $m$ suffices. When $m\le0$ we have $-r\le z\le0$, so $L(z)=1$ by
\eqref{eq:n-rec}, and $r+z\ge0$ gives
\[
        L(z)=1\le e^{\lstar(r+z)}=K_{\ell,r}e^{\lstar z},
\]
the base of the induction. Fix $m\ge1$ and assume $Q(m')$ for all $m'<m$. Any
$z$ at level $m$ has $z>0$, and each child $y\in\{z-\ell,z-r\}$ satisfies
$y>-r$ and lies below level $m$, so the hypothesis applies to both. Using
\eqref{eq:n-rec}, the hypothesis, and then \eqref{eq:lambda-star},
\[
        L(z)=L(z-\ell)+L(z-r)
        \le K_{\ell,r}e^{\lstar z}\bigl(e^{-\lstar\ell}+e^{-\lstar r}\bigr)
        =K_{\ell,r}e^{\lstar z},
\]
which is $Q(m)$. Taking logarithms in $e^{\lstar z}\le L(z)\le
K_{\ell,r}e^{\lstar z}$ finishes the proof.
\end{proof}

\section{A logarithmic threshold for cut schedules}
\label{sec:threshold}

With the branching rate in hand, we now bring cuts into the
picture. We begin with a normal form. For the node-count
objective, cuts may be moved to the root without increasing the
tree size. Analogous root-cut reductions are established for constant cuts
in~\cite[Lemma~7]{kazachkov2024abstract} and stated for
harmonically-worsening cuts therein, while the optimality of root cuts
under general node-processing-time functions is studied
in~\cite{han2026strength}. For completeness, we prove the version
needed here under the standing assumption $c_i\ge0$.

\begin{lemma}
\label{lem:root-cut-reduction}
For every target bound $Z$, there exists a minimal-size B\&C tree
proving $Z$ in which all cut nodes form a path starting at the root.
\end{lemma}

\begin{proof}
Let \(T\) be a minimal-size tree proving \(Z\), and let \(m\) be its
total number of cut nodes. Let \(B\) be the tree whose nodes are the
branch nodes and leaves of \(T\). The parent of such a node is its
nearest strict ancestor in \(T\) that is a branch node; the root of
\(B\) is the unique such node with no branch-node ancestor. This node is
unique because cut nodes have one child, so the portion of \(T\) before
the first branch node or leaf is a chain. Then \(B\) is a pure
branching tree with the same branch nodes as \(T\), and
since exactly the \(m\) cut nodes were deleted,
\(|V_B|=|V_T|-m\). Let \(T'\) consist of a path of \(m\) cut nodes
starting at the root, followed by \(B\). Then
\(|V_{T'}|=m+|V_B|=|V_T|\).

Consider any leaf \(v\) in \(T'\). Its branching improvements in \(T'\) are
those in \(T\), since \(v\) has the same branch-node ancestors in
both trees. Suppose the root-to-\(v\) path in \(T\) contained
\(k\) cut nodes. These are linearly ordered along the path, so the
\(j\)-th of them has exactly \(j-1\) cut nodes strictly above it.
Hence the cutting improvement at \(v\) in \(T\) was
\(c_1+\cdots+c_k=C_k\), while in \(T'\) it is
\(c_1+\cdots+c_m=C_m\). Moreover \(k\le m\), since the cut nodes on
one path form a subset of all cut nodes in \(T\). Because \(c_i\ge0\),
\(C_m-C_k=\sum_{i=k+1}^m c_i\ge0\), so the bound at \(v\) in \(T'\)
is at least its bound in \(T\). Therefore \(T'\) proves \(Z\), has
minimal size, and its cut nodes form a path starting at the root.
\end{proof}

Consequently, by Lemma~\ref{lem:root-cut-reduction}, the minimal size
over all B\&C trees is obtained by optimizing over \emph{cut-and-branch
trees}. In particular, after \(k\) root cuts the residual is \(Z-C_k\), and the pure
branching part has \(L(Z-C_k)\) leaves and \(2L(Z-C_k)-1\) nodes. Thus
\begin{equation}
        S^*_{\Bac}(Z)
        =
        \min_{k\ge0}
        \left\{k+2L(Z-C_k)-1\right\}.
        \label{eq:nested-size}
\end{equation}
Since $L(\cdot)\ge1$, the objective in \eqref{eq:nested-size} tends to infinity
with $k$, so the minimum is attained at some finite $k$.

In a cut-and-branch tree the target thus splits into a portion $C_k$
removed by root cuts and a residual $Z-C_k$ certified by branching.
Theorem~\ref{thm:cut-price-classification} is stated for schedules whose
extended limit \(\gamma=\lim_{k\to\infty}C_k/\log k\in[0,\infty]\)
exists, and shows  the limiting share of the target proved by cuts is characterized by $\gamma$.

We prove Theorem~\ref{thm:cut-price-classification} in two steps, isolating the
upper-bound direction first. The following lemma bounds the exponential size
rate from above by choosing the number of root cuts suitably in each case of
$\gamma$; the matching lower bound and the resulting split of the target are
then extracted in the proof of the theorem itself.

\begin{lemma}
\label{lem:size-rate-upper}
Assume that the extended limit
$\gamma=\lim_{k\to\infty}C_k/\log k\in[0,\infty]$ exists.
Then
\[
        \limsup_{Z\to\infty}\frac{\log S^*_{\Bac}(Z)}{Z}
        \le\frac{\lstar}{1+\gamma\lstar},
\]
where the value on the right is taken to be $0$ when $\gamma=\infty$.
\end{lemma}

\begin{proof}
Recall from \eqref{eq:nested-size} that
$S^*_{\Bac}(Z)=\min_{k\ge0}\Phi_Z(k)$ with $\Phi_Z(k)=k+2L(Z-C_k)-1$.
By Theorem~\ref{thm:branching-rate} we have $L(z)\le
K_{\ell,r}e^{\lstar z_+}$ for every real $z$, where $z_+=\max\{z,0\}$; the case $z\le0$ uses $L(z)=1\le K_{\ell,r}=e^{\lstar r}$.
Evaluating $\Phi_Z$ at an arbitrary $k\ge1$ and applying this bound gives
$S^*_{\Bac}(Z)\le k+2K_{\ell,r}e^{\lstar(Z-C_k)_+}$. Since $u+v\le
2\max\{u,v\}$ for nonnegative $u,v$, taking logarithms yields, for every
$k\ge1$,
\begin{equation}\label{eq:proof-upper-master}
        \frac{\log S^*_{\Bac}(Z)}{Z}
        \le
        \frac{\log2}{Z}
        +\max\!\left\{\frac{\log k}{Z},\;
        \frac{\log(2K_{\ell,r})}{Z}+\lstar\frac{(Z-C_k)_+}{Z}\right\}.
\end{equation}

Let $a:=\lstar/(1+\gamma\lstar)$ and $\rho:=1/(1+\gamma\lstar)$, so that $\lstar\rho=a$ and
$1-\rho=\gamma\lstar/(1+\gamma\lstar)$, the latter equal to $\gamma a$
when $\gamma$ is finite, while $a=\rho=0$ when $\gamma=\infty$. We bound
the left side of \eqref{eq:proof-upper-master} by choosing $k$ suitably
for $\gamma <\infty$ and $\gamma = \infty$ respectively.

Suppose first that $\gamma<\infty$, so that $a>0$. Take
$k=\tilde k(Z)=\lfloor e^{aZ}\rfloor$, which tends to infinity with $Z$.
Then $\log\tilde k(Z)/Z\to a$, and since $\tilde k(Z)\to\infty$,
\[
        \frac{C_{\tilde k}}{Z}
        =\frac{C_{\tilde k}}{\log\tilde k}\cdot\frac{\log\tilde k}{Z}
        \to\gamma a,
        \qquad\text{hence}\qquad
        \frac{(Z-C_{\tilde k})_+}{Z}\to1-\gamma a=\rho.
\]
Passing to the limit in \eqref{eq:proof-upper-master} along this choice
of $k=\tilde k(Z)$, and using $\lstar\rho=a$, gives
\[
        \limsup_{Z\to\infty}\frac{\log S^*_{\Bac}(Z)}{Z}
        \le\max\{a,\lstar\rho\}=a.
\]

Suppose instead that $\gamma=\infty$, so that $a=0$. Fix
$\eps\in(0,1)$ and let $\bar k(Z)$ be the least $k$ with
$C_k\ge(1-\eps)Z$, which exists because $C_k\to\infty$ as $k\to\infty$.
Fix $\eta>0$. Since $C_k/\log k\to\infty$, we have
$C_j\ge(1-\eps)\eta^{-1}\log j$ for sufficiently large $j$, and the choice
$j=\lceil e^{\eta Z}\rceil$ then forces $C_j\ge(1-\eps)Z$ once $Z$ is
large enough, so that $\bar k(Z)\le j$ and
$\limsup_{Z\to\infty}\log\bar k(Z)/Z\le\eta$. As $\eta$ was arbitrary,
$\limsup_{Z\to\infty}\log\bar k(Z)/Z\le0$. By the definition of
$\bar k(Z)$ the residual obeys $(Z-C_{\bar k})_+\le\eps Z$, so evaluating
\eqref{eq:proof-upper-master} at $k=\bar k(Z)$ gives
$\limsup_{Z\to\infty}\log S^*_{\Bac}(Z)/Z\le\lstar\eps$. Letting
$\eps\downarrow0$ shows that this limit superior is at most $0=a$. In
both cases the limit superior is at most $a$, which finishes the proof.
\end{proof}

With the upper bound of Lemma~\ref{lem:size-rate-upper} in hand, we now combine
it with the branching lower bound of Theorem~\ref{thm:branching-rate} to pin
down both the residual split $z^*(Z)_+/Z$ and the exact exponential size rate.

\begin{theorem}
\label{thm:cut-price-classification}
Assume that the extended limit
\(\gamma=\lim_{k\to\infty}C_k/\log k\in[0,\infty]\) exists.
Let $k^*(Z)$ be any minimizer of \eqref{eq:nested-size}, set
$z^*(Z)=Z-C_{k^*(Z)}$, and define
$\widehat f(Z)=\min\{C_{k^*(Z)},Z\}/Z=1-z^*(Z)_+/Z$, where
$z_+\defeq\max\{z,0\}$. Then, as $Z\to\infty$,
$z^*(Z)_+/Z\to1/(1+\gamma\lstar)$ and
$\widehat f(Z)\to\gamma\lstar/(1+\gamma\lstar)$. Moreover,
\begin{equation}
        \frac{\log S^*_{\Bac}(Z)}{Z}\to\frac{\lstar}{1+\gamma\lstar}.
        \label{eq:bac-size-rate}
\end{equation}
When $\gamma=\infty$, these three limits are interpreted as $0,1,0$,
respectively. In particular, for $\gamma=0$, branch-and-cut and pure
branch-and-bound have the same exponential size rate $\lstar$, since
$2L(Z)-1=\Theta_{\ell,r}(e^{\lstar Z})$ by
Theorem~\ref{thm:branching-rate}.
\end{theorem}

\begin{proof}
Write $k^*=k^*(Z)$ and $z^*=z^*(Z)=Z-C_{k^*}$, and recall
$\Phi_Z(k)=k+2L(Z-C_k)-1$, so that
$S^*_{\Bac}(Z)=\Phi_Z(k^*)=\min_{k\ge0}\Phi_Z(k)$ by
\eqref{eq:nested-size}. As in Lemma~\ref{lem:size-rate-upper}, let $a:=\lstar/(1+\gamma\lstar)$ and $\rho:=1/(1+\gamma\lstar)$, with $a=\rho=0$ when $\gamma=\infty$, and
recall $\lstar\rho=a$ together with
$1-\rho=\gamma\lstar/(1+\gamma\lstar)$. The three assertions to be proved
are $z^*_+/Z\to\rho$, then $\widehat f(Z)\to1-\rho$, and
$\log S^*_{\Bac}(Z)/Z\to a$ as $Z \to \infty$.

Theorem~\ref{thm:branching-rate} supplies the two branching estimates
\[
        L(z)\le K_{\ell,r}\,e^{\lstar z_+}\quad(z\in\R),
        \qquad
        L(z)\ge e^{\lstar z}\quad(z>0),
\]
the first covering $z\le0$ as well, since there $L(z)=1\le
K_{\ell,r}=e^{\lstar r}$. Evaluating $\Phi_Z$ at $k^*$ and using $L\ge1$
and $k^*\ge0$, we derive the two lower estimates
\begin{equation}\label{eq:proof-lower-master}
        k^*+1\le S^*_{\Bac}(Z),
        \qquad
        e^{\lstar z^*_+}\le 2L(z^*)-1\le S^*_{\Bac}(Z).
\end{equation}
The middle inequality of the second part holds for $z^*>0$ by the lower
branching estimate and for $z^*\le0$ because then $L(z^*)=1$ and
$z^*_+=0$.

Lemma~\ref{lem:size-rate-upper} already gives
$\limsup_{Z\to\infty}\log S^*_{\Bac}(Z)/Z\le a$, and we combine it with
\eqref{eq:proof-lower-master} to prove the residual limit
$z^*_+/Z\to\rho$. From the second estimate in
\eqref{eq:proof-lower-master} we have $\lstar z^*_+\le\log
S^*_{\Bac}(Z)$, so dividing by $Z$ and $\lstar$ and passing to the limit superior gives
\[
        \limsup_{Z\to\infty}\frac{z^*_+}{Z}\le\frac{a}{\lstar}=\rho.
\]
When $\gamma=\infty$ this already forces $z^*_+/Z\to0=\rho$. Assume
henceforth that $\gamma<\infty$, and bound $C_{k^*}$ from above so as to
bound $z^*_+$ from below. The first estimate in
\eqref{eq:proof-lower-master} gives $\log(k^*+1)\le\log S^*_{\Bac}(Z)$,
hence $\limsup_{Z\to\infty}\log(k^*+1)/Z\le a$. Fix $\theta>0$ and
choose $M\ge2$ with $C_k\le(\gamma+\theta)\log k$ for all $k\ge M$.
Whether $k^*<M$, in which case $C_{k^*}\le C_{M-1}$, or $k^*\ge M$, in
which case $C_{k^*}\le(\gamma+\theta)\log(k^*+1)$, we have
\[
        C_{k^*}\le C_{M-1}+(\gamma+\theta)\log(k^*+1).
\]
The constant $C_{M-1}$ is independent of $Z$, so $C_{M-1}/Z\to0$. Combining it with $\limsup_{Z\to\infty}\log(k^*+1)/Z\le a$ yields 
$\limsup_{Z\to\infty}C_{k^*}/Z\le(\gamma+\theta)a$. Letting
$\theta\downarrow0$ gives $\limsup_{Z\to\infty}C_{k^*}/Z\le\gamma
a=1-\rho$. Since $z^*_+\ge z^*=Z-C_{k^*}$, we conclude
\[
        \liminf_{Z\to\infty}\frac{z^*_+}{Z}\ge1-(1-\rho)=\rho,
\]
and with the matching limit superior this proves $z^*_+/Z\to\rho$.

The two remaining assertions now follow. By the definition of
$\widehat f$,
\[
        \widehat f(Z)=1-\frac{z^*_+}{Z}\to1-\rho
        =\frac{\gamma\lstar}{1+\gamma\lstar}.
\]
For the size rate, the second estimate in \eqref{eq:proof-lower-master}
gives $\log S^*_{\Bac}(Z)/Z\ge\lstar z^*_+/Z$, whose limit is
$\lstar\rho=a$, so $\liminf_{Z\to\infty}\log S^*_{\Bac}(Z)/Z\ge a$. With
the matching bound from Lemma~\ref{lem:size-rate-upper} this gives
$\log S^*_{\Bac}(Z)/Z\to a=\lstar/(1+\gamma\lstar)$, which is
\eqref{eq:bac-size-rate}. In particular, when $\gamma=0$ the rate equals
$\lstar$, matching the pure $\Bab$ rate.
\end{proof}

We now specialize Theorem~\ref{thm:cut-price-classification} to the
two cut schedules studied in~\cite{kazachkov2024abstract} via
Corollary~\ref{cor:recoveries}.  For constant cuts, the corollary recovers,
and extends to $\ell\neq r$, a conclusion
of~\cite[Theorem~9]{kazachkov2024abstract} that the proportion of the
bound proved by branching vanishes as $Z\to\infty$.  For
harmonically-worsening cuts, the situation is more subtle.
Theorem~21 of~\cite{kazachkov2024abstract} establishes the limit fraction
$c\log2/(r+c\log2)$ for the approximately optimal trees produced by their Algorithm~1, whose
approximation guarantee carries a multiplicative factor, and
whether minimal-size trees exhibit the same limit is investigated
computationally in~\cite[Appendix~B]{kazachkov2024abstract}.
Since the trichotomy applies to an exact minimizer $k^*(Z)$
of~\eqref{eq:nested-size}, the corollary proves that the same limit
holds for minimal-size trees, confirming the observed convergence.

\begin{corollary}
\label{cor:recoveries}
Let $0<\ell\le r$ and $c>0$.
\begin{enumerate}
\item For constant cuts $C_k=ck$, we have 
$\widehat f(Z)\to1$ and $z^*(Z)_+/Z\to0$.
\item For harmonically-worsening cuts $C_k=cH_k$, where
$H_k=\sum_{i=1}^{k}1/i$ is the $k$th harmonic number, we have
\(
        \widehat f(Z)\longrightarrow{c\lstar}/{(1+c\lstar)},
\)
which for $\ell=r$, where $\lstar=\log2/r$, equals
$c\log2/(r+c\log2)$.
\end{enumerate}
\end{corollary}

\begin{proof}
Both schedules satisfy the standing assumption $c_i\ge0$. For
$C_k=ck$ we have $C_k/\log k\to\infty$, while for $C_k=cH_k$ the
bound $\log(k+1)<H_k\le\log k+1$ gives $C_k/\log k\to c$.  Applying
Theorem~\ref{thm:cut-price-classification} with $\gamma=\infty$ and
$\gamma=c$ respectively yields the claims; when $\ell=r$, \eqref{eq:lambda-star}
gives $\lstar=\log2/r$.
\end{proof}

\section{A polynomial versus exponential dichotomy}
\label{sec:dichotomy}

Section~\ref{sec:threshold} concerned the asymptotic share of the
target certified by cuts. We now ask a complementary question: when do
cuts change the complexity class of the proof itself, reducing the tree
size from exponential in $Z$ to polynomial? 

Let $S^*_{\Bab}(Z)\defeq2L(Z)-1$ denote the minimal pure $\Bab$ size.
By Theorem~\ref{thm:branching-rate},
$S^*_{\Bab}(Z)=\Theta_{\ell,r}(e^{\lstar Z})$, so any speedup must
come from cuts. Lemma~\ref{lem:root-cut-reduction} and the size formula
\eqref{eq:nested-size} apply under the standing assumption $c_i\ge0$. A
cut-and-branch tree with $k$ cuts has size
$k+\Theta_{\ell,r}(e^{\lstar(Z-C_k)_+})$, so a polynomial total forces
both $k=Z^{O(1)}$ and $(Z-C_k)_+=O(\log Z)$. The following theorem shows
these necessary conditions are jointly sufficient.

\begin{theorem}
\label{thm:dichotomy}
For any cut schedule in the model, $S^*_{\Bac}(Z)=Z^{O(1)}$ if and only
if there is a polynomially bounded cut count $k(Z)\le Z^{O(1)}$ such that
\begin{equation}
        (Z-C_{k(Z)})_+=O(\log Z).
        \label{eq:log-residual-condition}
\end{equation}
\end{theorem}

\begin{proof}
We characterize when $S^*_{\Bac}(Z)$ is polynomial using
\eqref{eq:nested-size}.

\smallskip
\noindent$(\Leftarrow)$ Suppose $k(Z)\le Z^{O(1)}$ satisfies
\eqref{eq:log-residual-condition}.  Evaluating \eqref{eq:nested-size} at
this cut count and applying the extension $L(z)\le K_{\ell,r}e^{\lstar z_+}$
noted in the proof of Lemma~\ref{lem:size-rate-upper},
\(
        S^*_{\Bac}(Z)
        \le
        k(Z)+2L(Z-C_{k(Z)})-1
        \le
        Z^{O(1)}+O_{\ell,r}\!\left(e^{\lstar(Z-C_{k(Z)})_+}\right)
        =
        Z^{O(1)},
\)
where the last equality uses \eqref{eq:log-residual-condition} in the form
$(Z-C_{k(Z)})_+=O(\log Z)$.
Hence B\&C has polynomial size.

\smallskip
\noindent$(\Rightarrow)$ Suppose $S^*_{\Bac}(Z)\le Z^{a}$ for some constant
$a$, and let $k^*(Z)$ minimize \eqref{eq:nested-size};
such a minimizer exists by the sentence following \eqref{eq:nested-size}.
Since $k^*(Z)+2L(Z-C_{k^*(Z)})-1\le Z^{a}$ with both terms nonnegative, we
have $k^*(Z)\le Z^{a}$, so the cut count is polynomially bounded, and
$L(Z-C_{k^*(Z)})\le Z^{a}$.  When $Z-C_{k^*(Z)}>0$, the lower bound
$e^{\lstar z}\le L(z)$ of Theorem~\ref{thm:branching-rate} gives
$e^{\lstar(Z-C_{k^*(Z)})}\le Z^{a}$, so taking logarithms,
\(
        (Z-C_{k^*(Z)})_+
        \le
        \frac{a}{\lstar}\log Z,
\)
the case $Z-C_{k^*(Z)}\le0$ being trivial.  Thus
\eqref{eq:log-residual-condition} holds with $k(Z)=k^*(Z)$.
\end{proof}

Theorem~\ref{thm:dichotomy} complements the results of Basu et al.~\cite{basu2022complexity}, who
exhibit instances on which branch-and-cut is exponentially smaller
than pure branching or pure cutting alone.  The next corollary
locates the boundary of this speedup on the scale of cut schedules. In particular, 
polynomial growth of $C_k$ suffices, while harmonically-worsening
cuts, despite proving a constant fraction of the target by
Corollary~\ref{cor:recoveries}, never achieve it.

\begin{corollary}
\label{cor:poly-schedules}
If $C_k\ge\beta k^{\alpha}$ for some constants $\alpha,\beta>0$, then $S^*_{\Bac}(Z)=Z^{O(1)}$.  In contrast, for
harmonically-worsening cuts $C_k=cH_k$, every cut count $k(Z)$ with
$(Z-C_{k(Z)})_+=O(\log Z)$ satisfies $k(Z)=e^{\Omega(Z)}$, so
$S^*_{\Bac}(Z)$ is not polynomial in $Z$.
\end{corollary}

\begin{proof}
For the first claim, take $k(Z)=\lceil(Z/\beta)^{1/\alpha}\rceil$;
then $C_{k(Z)}\ge Z$, so \eqref{eq:log-residual-condition} holds
with a polynomially bounded cut count, and
Theorem~\ref{thm:dichotomy} applies.  For the second claim, suppose $k(Z)$ satisfies
\eqref{eq:log-residual-condition}, so that
$C_{k(Z)}\ge Z-O(\log Z)$.  Since $H_k\le\log k+1$, we also have
$C_{k(Z)}\le c\bigl(\log k(Z)+1\bigr)$.  Therefore
$c\bigl(\log k(Z)+1\bigr)\ge Z-O(\log Z)$, i.e.,
$k(Z)=e^{\Omega(Z)}$, so no polynomially bounded cut count
satisfies \eqref{eq:log-residual-condition}.
\end{proof}

\section{Concluding remarks}
\label{sec:conclusion}

We studied minimal tree size in the single-variable abstract B\&C model with
asymmetric branching. The branching rule $(\ell,r)$ enters only through the
rate $\lstar$ defined by $e^{-\lstar\ell}+e^{-\lstar r}=1$, and the
logarithmic cumulative cut strength is critical. If $C_k$ grows faster than
$\log k$, cuts prove asymptotically all of the target. If $C_k$ is asymptotic
to $\gamma\log k$, cutting and branching each prove a constant share, with
cut fraction $\gamma\lstar/(1+\gamma\lstar)$. If $C_k$ grows slower than
$\log k$, cuts yield no exponential-size improvement over pure B\&B.
Moreover, B\&C is polynomial while pure B\&B is exponential exactly when
polynomially many cuts reduce the residual bound to $O(\log Z)$.

Several questions remain open. Our trichotomy assumes that the extended
limit $\gamma=\lim_{k\to\infty}C_k/\log k$ exists, and it would be
interesting to describe the behavior of minimal trees when $C_k/\log k$
oscillates. It is also natural to ask whether analogous thresholds hold for
general time functions beyond node count, for branching rules whose
improvements vary across nodes, and for versions of the model with multiple
variables, in the spirit of the multiple- and general-variable branching
problems of Le Bodic and Nemhauser~\cite{le2017abstract}.

\medskip

\noindent \textbf{Acknowledgments.}{ The author declares
no competing interests. No data were generated or
analyzed in this study. ChatGPT from OpenAI was used to refine the original proofs, and the author verified
and takes full responsibility for the final draft.}

\bibliographystyle{plain}
\bibliography{references}

@article{land1960automatic,
  title = {An Automatic Method of Solving Discrete Programming Problems},
  author = {Land, A. H. and Doig, A. G.},
  journal = {Econometrica},
  volume = {28},
  number = {3},
  pages = {497--520},
  year = {1960},
  month = {Jul},
  publisher = {The Econometric Society},
  url = {https://www.jstor.org/stable/1910129}
}

@article{gomory1958outline,
  title = {Outline of an algorithm for integer solutions to linear programs},
  author = {Gomory, Ralph E.},
  journal = {Bulletin of the American Mathematical Society},
  volume = {64},
  number = {5},
  pages = {275--278},
  year = {1958},
  month = {Sep}
}

@book{junger200950,
  title={50 Years of Integer Programming 1958--2008: From the Early Years to the State-of-the-Art},
  editor={J{\"u}nger, Michael and Liebling, Thomas M. and Naddef, Denis and Nemhauser, George L. and Pulleyblank, William R. and Reinelt, Gerhard and Rinaldi, Giovanni and Wolsey, Laurence A.},
  year={2009},
  publisher={Springer},
  address={Berlin, Heidelberg},
  isbn={978-3-540-68274-5},
  edition={1},
  pages={804}
}

@incollection{achterberg2013mixed,
  author    = {Achterberg, Tobias and Wunderling, Roland},
  title     = {Mixed Integer Programming: Analyzing 12 Years of Progress},
  booktitle = {Facets of Combinatorial Optimization: Festschrift for Martin Gr{\"o}tschel},
  editor    = {J{\"u}nger, Michael and Reinelt, Gerhard},
  publisher = {Springer},
  address   = {Berlin, Heidelberg},
  year      = {2013},
  pages     = {449--481},
  isbn      = {978-3-642-38188-1}
}

@article{le2017abstract,
  title={An abstract model for branching and its application to mixed integer programming},
  author={Le Bodic, Pierre and Nemhauser, George},
  journal={Mathematical Programming},
  volume={166},
  number={1},
  pages={369--405},
  year={2017},
  publisher={Springer}
}

@article{anderson2021further,
  title={Further results on an abstract model for branching and its application to mixed integer programming},
  author={Anderson, Daniel and Le Bodic, Pierre and Morgan, Kerri},
  journal={Mathematical Programming},
  volume={190},
  number={1},
  pages={811--841},
  year={2021},
  publisher={Springer}
}

@inproceedings{dadush2020complexity,
  title={On the Complexity of Branching Proofs},
  author={Dadush, Daniel and Tiwari, Samarth},
  booktitle={35th Computational Complexity Conference (CCC 2020)},
  pages={34--1},
  year={2020},
  organization={Schloss Dagstuhl--Leibniz-Zentrum f{\"u}r Informatik}
}

@article{dey2023lower,
  title={Lower bounds on the size of general branch-and-bound trees},
  author={Dey, Santanu S and Dubey, Yatharth and Molinaro, Marco},
  journal={Mathematical Programming},
  volume={198},
  number={1},
  pages={539--559},
  year={2023},
  publisher={Springer}
}

@article{dey2024theoretical,
  title={A theoretical and computational analysis of full strong-branching},
  author={Dey, Santanu S and Dubey, Yatharth and Molinaro, Marco and Shah, Prachi},
  journal={Mathematical Programming},
  volume={205},
  number={1},
  pages={303--336},
  year={2024},
  publisher={Springer}
}

@article{dey2018theoretical,
  title={Theoretical challenges towards cutting-plane selection},
  author={Dey, Santanu S and Molinaro, Marco},
  journal={Mathematical Programming},
  volume={170},
  number={1},
  pages={237--266},
  year={2018},
  publisher={Springer}
}

@article{balas2010enumerative,
  title={On the enumerative nature of {Gomory}'s dual cutting plane method},
  author={Balas, Egon and Fischetti, Matteo and Zanette, Arrigo},
  journal={Mathematical Programming},
  volume={125},
  number={2},
  pages={325--351},
  year={2010},
  publisher={Springer}
}

@article{zanette2011lexicography,
  title={Lexicography and degeneracy: can a pure cutting plane algorithm work?},
  author={Zanette, Arrigo and Fischetti, Matteo and Balas, Egon},
  journal={Mathematical Programming},
  volume={130},
  number={1},
  pages={153--176},
  year={2011},
  publisher={Springer}
}

@article{dey2022cutting,
  title={Cutting plane generation through sparse principal component analysis},
  author={Dey, Santanu S and Kazachkov, Aleksandr and Lodi, Andrea and Munoz, Gonzalo},
  journal={SIAM Journal on Optimization},
  volume={32},
  number={2},
  pages={1319--1343},
  year={2022},
  publisher={SIAM}
}

@article{berthold2025learning,
  title={Learning to use local cuts},
  author={Berthold, Timo and Francobaldi, Matteo and Hendel, Gregor},
  journal={Mathematical Programming Computation},
  volume={17},
  number={3},
  pages={437--450},
  year={2025},
  publisher={Springer}
}

@article{basu2023complexity,
  title={Complexity of branch-and-bound and cutting planes in mixed-integer optimization},
  author={Basu, Amitabh and Conforti, Michele and Di Summa, Marco and Jiang, Hongyi},
  journal={Mathematical Programming},
  volume={198},
  number={1},
  pages={787--810},
  year={2023},
  publisher={Springer Berlin Heidelberg Berlin/Heidelberg}
}

@article{shah2026non,
  title={Non-monotonicity of branching rules with respect to linear relaxations},
  author={Shah, Prachi and Dey, Santanu S and Molinaro, Marco},
  journal={INFORMS Journal on Computing},
  volume={38},
  number={1},
  pages={53--66},
  year={2026},
  publisher={INFORMS}
}

@inproceedings{han2026strength,
  title={The Strength of Root Cuts in an Extended Abstract Branch-and-Cut Model},
  author={Han, Boyang and Kazachkov, Aleksandr M},
  booktitle={International Conference on Integer Programming and Combinatorial Optimization},
  pages={474--490},
  year={2026},
  organization={Springer}
}

@article{kazachkov2024abstract,
  title={An abstract model for branch and cut},
  author={Kazachkov, Aleksandr M and Le Bodic, Pierre and Sankaranarayanan, Sriram},
  journal={Mathematical Programming},
  volume={206},
  number={1},
  pages={175--202},
  year={2024},
  publisher={Springer}
}

@article{basu2022complexity,
  title={Complexity of branch-and-bound and cutting planes in mixed-integer optimization---{II}},
  author={Basu, Amitabh and Conforti, Michele and Di Summa, Marco and Jiang, Hongyi},
  journal={Combinatorica},
  volume={42},
  number={Suppl 1},
  pages={971--996},
  year={2022},
  publisher={Springer}
}

@article{cheng2026generalization,
  title={Generalization guarantees for learning score-based branch-and-cut policies in integer programming},
  author={Cheng, Hongyu and Basu, Amitabh},
  journal={Advances in Neural Information Processing Systems},
  volume={38},
  pages={118669--118699},
  year={2026}
}

\end{document}